\documentclass{article}
\usepackage{amsmath}
\usepackage{amsfonts}
\usepackage{titlesec}
\usepackage{amsthm}

\newtheorem{theorem}{Theorem}[section]

\author{George Robinson}
\title{Reduced 2-coloured Khovanov Homology detects the Trefoil}
\date{}

\titleformat{\section}
{\normalfont\bfseries}{\quad\thesection .}{1em}{}

\begin{document}
	
\maketitle
\section{Introduction}
The question of which knots are detected by the Jones polynomial and its coloured variants is still very much open, but recently several weaker questions have been answered regarding a categorification of the Jones polynomial, first introduced by Khovanov in \cite{categorification}. In \cite{khovanov}, Khovanov also introduced categorifications for two versions of the $n$-coloured Jones polynomial, one returning $[n+1]$ for the unknot, the other returning $1$.

The first major detection result concerning these constructions came in \cite{spectralsequence}, where Grigsy and Wehrli constructed a spectral sequence from the n-coloured reduced categorification to knot Floer homology.

\begin{theorem}
	Let $K\subset S^3$ be an oriented knot, $\bar{K}\subset S^3$, and $K^r$ its orientation reverse. There is a spectral sequence whose $E^2$ term is $\widetilde{Kh}_2(\bar{K})$ and whose $E^{\infty}$ term is $\widehat{HFK}(S^3,K\#K^r)$.
\end{theorem}

This leads to the easy corollary that the reduced 2-coloured Khovanov homology detects the unknot (and in fact it generalises to all $n>1$). In the unreduced case for $n=2$, Hedden proved unknot detection in \cite{cable}, using a spectral sequence from the Khovanov homology to the Floer homology of the branched double cover of a knot, first noted by Ozsv\'ath and Szab\'o in \cite{branched}.

It was then shown by Kronheimer and Mrowka in \cite{unknotdetect} that Khovanov homology detects the unknot, and by Baldwin and Sivek in \cite{trefoildetect} that it detects the trefoils.
 
The main result of this note is Theorem \ref{detection}, where we prove that the reduced 2-coloured Khovanov homology also detects the trefoil.

\section{Result}
\begin{theorem}
\label{detection}
	Let $K$ be a knot in $S^3$. Then $K$ is the trefoil if and only if $\text{\normalfont rk\ }\widetilde{\text{Kh}}_2(K)=9$, where $\widetilde{\text{Kh}}_2(K)$ denotes the reduced 2-coloured Khovanov homology of $K$.
\end{theorem}
\begin{proof}
	Let $K$ denote the (right-handed) trefoil in $S^3$. A computer calculation (the code for which is available on GitHub\footnote{https://github.com/robinsongeorge/Reduced-Khovanov-Homology}) of the rank of $\widetilde{Kh}_2(K)$ gives 9.
	
	For the converse implication, suppose that $K\subset S^3$ is a knot such that $\text{\normalfont rk\ }\widetilde{\text{Kh}}_2(K)\le24$. In particular, from the spectral sequence in \cite{spectralsequence}, it follows that
	\begin{align*}
	\left(\text{rk\ }\widehat{HFK}(K)\right)^2&=\text{rk\ }\widehat{HFK}(S^3,K\#K^r) \\
	&\le\text{rk\ }\widetilde{Kh}_2(K) \\
	&\le24.
	\end{align*}
	Combining this with the fact that $\text{rk\ }\widehat{HFK}(K)$ is odd, we see that $\text{rk\ }\widehat{HFK}(K)=1 \text{\ or\ } 3$. As noted in \cite{genus}, when $\text{rk\ }\widehat{HFK}(K)=1$ the fact that $\widehat{HFK}(K)$ is symmetric in the Alexander grading implies it must be supported in grading 0, and then since $\widehat{HFK}(K)$ detects genus, $K$ must be the unknot. In the case $\text{rk\ }\widehat{HFK}(K)=3$, $K$ must be a trefoil, as shown in \cite{hedden} by Hedden and Watson.
\end{proof}
The fact that $\text{rk\ }\widetilde{\text{Kh}}_2(K)=9$ for the trefoil, and $\text{rk\ }\widetilde{\text{Kh}}_2(K)=25$ for the figure-eight knot suggests that perhaps the spectral sequence always collapses by the $E^2$ page for alternating knots, however this is not true for links since the 2-4 torus link has $\text{rk\ }\widetilde{\text{Kh}}_2(L)=18$.


\begin{thebibliography}{10}
	\bibitem{trefoildetect}
	John A. Baldwin and Steven Sivek.
	\textit{Khovanov homology detects the trefoils}
	arxiv.org/abs/1801.07634, 2018.
	\bibitem{spectralsequence}
	J. Elisenda Grigsby and Stephan M. Wehrli.
	\textit{On the colored Jones polynomial, sutured Floer homology, and knot Floer homology}.
	Advanced in Mathematics, 223:2114-2165, 2009.
	\bibitem{cable}
	Matthew Hedden.
	\textit{Khovanov homology of the 2-cable detects the unknot}
	math.GT/0805.4418, 2008.
	\bibitem{hedden}
	Matthew Hedden and Liam Watson.
	\textit{On the geography and botany of knot Floer homology}
	arXiv:1404.6913, 2014.
	\bibitem{categorification}
	Mikhail Khovanov.
	\textit{A categorification of the Jones polynomial}
	Duke Math. J. , 101(3):359–426, 2000.
	\bibitem{khovanov} 
	Mikhail Khovanov.
	\textit{Categorifications of the colored Jones polynomial}. 
	J. Knot Theory Ramifications, 14(1):111-130, 2005
	\bibitem{unknotdetect}
	P. Kronheimer and T. Mrowka.
	\textit{Khovanov homology is an unknot detector}.
	Publ. Math. Inst.Hautes ́Etudes Sci., (113):97–208, 2011.
	\bibitem{genus}
	Peter Ozsv\'ath and Zoltan Szab\'o.
	\textit{Holomorphic disks and genus bounds}
	Geom. Topol. 8 (2004) 311-334, 2003.
	\bibitem{branched}
	Peter Ozsv\'ath and Zoltan Szab\'o.
	\textit{On the Heegaard Floer homology of branched double-covers}
	Adv. Math. 194 (2005), no. 1, 1--33.
\end{thebibliography}
\end{document}